\documentclass[12pt, a4paper, twosided, reqno]{amsart}
\usepackage{enumerate, cite}
\usepackage{hyperref}

\newtheorem{theorem}{Theorem}
\newtheorem{lemma}{Lemma}

\newtheorem{example}{Example}

\newtheorem*{thA}{Theorem A}

\allowdisplaybreaks
\author[N. Gahlian ]{Nidhi Gahlian }

\address{nidhi gahlian; department of mathematics, university of delhi, delhi-110007, india.}
\email{nidhigahlyan81@gmail.com}

\thanks {Research work of the author is supported by research fellowship from Department of Science and Technology(INSPIRE), New Delhi, India, IF-190674.}
\title[Study of Solutions ]{Study of Solutions of Certain type of Non-Linear Differential-Difference Equations}
\begin{document}
	\maketitle
	
	\begin{abstract}
		In this paper, we analyze the solutions of the following non-linear differential-difference equations
		$$f^{n}(z)+wf^{n-1}(z)f^{'}(z)+p(z)f(z+c)=p_{1}e^{\alpha_{1}z}+p_{2}e^{\alpha_{2}z}$$ and 
		$$f^{n}(z)f'(z)+q(z)e^{Q(z)}f(z+c)=p_{1}e^{\alpha_{1}z}+p_{2}e^{\alpha_{2} z},$$
		where $n$ is a positive integer, $w, p_{1}, p_{2}, \alpha_{1}$ $\&$  $\alpha_{2}$ are non-zero constants satisfying \qquad $\alpha_{1}$ $\neq$ $\alpha_{2}$,   $\frac{\alpha_1}{\alpha_2}\neq (n)^{\pm1}$, $q(z)$ is a non-vanishing polynomial and $Q(z)$ is a non-constant polynomial.
	\end{abstract}
	\section{\textbf{Introduction}}

	Many researchers have extensively explored the existence and solvability of entire and meromorphic solutions to nonlinear complex differential, difference, and differential–difference equations. A major development in this direction occurred in 1964, when Hayman \cite{hay} extended the Tumura–Clunie theorem, which led to a surge of interest in equations of Tumura–Clunie type, particularly those of the form $f^n(z) + Q_d(z,f) = g(z) $. Since then, these equations have been studied extensively, with many results obtained by considering different types of perturbations.
	
	As an illustrative example, Yang and Li \cite{yangcc} investigated the  differential equation
\begin{equation*}
	4f^3+3f''=-\sin 3z,
\end{equation*}
	and proved that it has exactly three non-constant entire solutions;
\begin{equation*}
	f_1(z) = \sin z,  \;	f_2(z) = -\frac{\sqrt{3}}{2}\cos z - \frac{1}{2} \sin z, \; \text{and} \; 
	f_3(z) = \frac{\sqrt{3}}{2}\cos z - \frac{1}{2}\sin z.
\end{equation*}
	Notably, the term $-\sin 3z $ can be written as
	\begin{equation*}
	-\sin 3z=\frac{1}{2\iota}(e^{-3\iota z}-e^{3\iota z}),
\end{equation*} 
	which expresses it as a linear combination of exponential functions. This observation naturally motivates the study of differential equations whose right-hand side consists of exponential terms.

In	particular, equations of the form
	\begin{equation*}
		f^n(z)+Q_d(z,f)=p_1e^{\alpha z}+p_2e^{-\alpha z}
	\end{equation*}
	 have attracted significant attention, and have been further  generalized  by replacing $f^n$ by $f^nf'$ or considering more general exponential terms and coefficients, including polynomials or rational functions, thus broadening the scope of investigation. One can see\cite{ chen, zh, yl, yangcc} and references therein. 
	It can be observed that the left-hand side of all these equations contains only one dominant term, either $f^n$ or $f^nf'$. Therefore, a natural and interesting direction for further investigation is to study equations that involve two dominant terms as in equation \eqref{oo}. Significant work has also been carried out in this direction (see \cite{ng, jh} and references therein).\\
	For instance, Li and Huang \cite{jh} studied the following equation
	  \begin{equation}\label{oo}
	 	f^n+\omega f^{n-1}f'+ f^k(z+c)=p_1e^{\alpha_1 z}+ p_2e^{\alpha_2 z},
	 \end{equation}
where  $\omega, p_1, p_2, \alpha_1, \alpha_2$ are non-zero constants and $k\geq1$ is an integer, with the restrictions  $\alpha_1 \neq \alpha_2$,  $\frac{\alpha_1}{\alpha_2}\neq n$ and $\frac{\alpha_2}{\alpha_1}\neq n$. Under these conditions, they showed that for $n\geq5$, equation \eqref{oo}  admits no transcendental solutions.
Motivated by this  literature, we now  establish the following result.
	\begin{theorem}\label{theoremA}
		Let $n\geq4$ be an integer, $p(z)$ be a non-vanishing polynomial and  $ p_1, p_2, \alpha_1, \alpha_2$ be non-zero constants such that $\alpha_1 \neq \alpha_2$ and $\frac{\alpha_1}{\alpha_2}\neq (n)^{\pm1}$. If the equation 
		\begin{equation}\label{1}
			f^n+\omega f^{n-1}f'+ p(z)f(z+c)=p_1e^{\alpha_1 z}+ p_2e^{\alpha_2 z},
		\end{equation}
admits a meromorphic solution $f(z)$ with $\rho_2 <1$, then $n=4$ and $f$ satisfies $\Bar{\lambda}(f)=\rho(f)$.
\end{theorem}
The next theorem is a refinement of Theorem \ref{theoremA}, obtained by incorporating an additional condition on the growth of the solution $f$.
\begin{theorem}
For $n=3$ and $4$, let  $p(z)$ be a non-vanishing polynomial and  $ p_1, p_2, \alpha_1$ and $ \alpha_2$ be non-zero constants such that $\alpha_1 \neq \alpha_2$ and $\frac{\alpha_1}{\alpha_2}\neq (n)^{\pm1}$. If ${\lambda}(f)<\rho(f)$, then equation \eqref{1}  does not admit any meromorphic solution $f(z)$. 
\end{theorem}
Wen et al. \cite{wen} studied and classified the finite  order solutions of the following equation,
$$ f^n(z)+q(z)e^{Q(z)}f(z+c)=P(z),$$
where $P,q,Q$ are polynomials, $n\geq2$ is an integer, and $c\in \mathbb{C}\setminus\{0\}$. 
In $2019$, Chen et al. \cite{ch} studied the above equation by replacing R.H.S. with $p_1 e^{\alpha z} + p_2 e^{-\alpha z}$. Motivated by this, many researchers  actively engage in extending existing theorems, and over time, the literature has witnessed the emergence of more elegant results by replacing R.H.S. with 
 $p_1 e^{\alpha_1z} + p_2 e^{\alpha_2z}$.
 In this sequence, Wang et al.\cite{wangy} proved the following result.
 \begin{thA}\cite{wangy}\label{ww}
	Let $n\geq4$ be an integer, $q(z)$ be a rational function and $Q(z)$ be a non-constant polynomial. Suppose that $c$, $p_1$, $p_2$, $\alpha_1$, and $\alpha_2\in \mathbb{C}\setminus\{0\}$ with $\alpha_1\neq \alpha_2$. If the equation 
	\begin{equation}\label{6}
		f^n(z)f'(z)+q(z)f(z+c)e^{Q(z)}=p_1e^{\alpha_1z}+p_2e^{\alpha_2z},
	\end{equation}
	admits a finite order meromorphic solution $f(z)$ with $\lambda(f) <\rho(f)$ and $\lambda\left(\frac{1}{f}\right) <\rho(f)$, then $f(z)$ satsfies $\rho(f)=\deg Q =1$.\;Furthermore, $f(z)=A_1e^{\frac{\alpha_2}{n+1}z}$,\linebreak $Q(z)=\left(\alpha_1-\frac{\alpha_2}{n+1}\right)z + a_0$
	or $f(z)=A_2e^{\frac{\alpha_1}{n+1}z}$, $Q(z)=\left(\alpha_2-\frac{\alpha_1}{n+1}\right)z+a_0$, where $A_1,A_2, a_o \in \mathbb{C}\setminus\{0\}$.
\end{thA}

In our previous work \cite{ng}, based on an example, we indicated that the given Theorem $A$ might hold for $n = 2,3$, but we lacked a full proof at that time. Now, we rigorously prove that the result holds for  $ n =1, 2, $ and $3 $ along with supporting examples.
	  
	  \begin{theorem}
	  	Suppose $n=1$, $2$ or $3$, $q(z)$  a rational function, $Q(z)$  a non-constant polynomial, and $c$, $p_1$, $p_2$, $\alpha_1$, $\alpha_2\in \mathbb{C}\setminus\{0\}$ with $\alpha_1\neq \alpha_2$. If the equation 
	  	\begin{equation}\label{z}
	  		f^n(z)f'(z)+q(z)f(z+c)e^{Q(z)}=p_1e^{\alpha_1z}+p_2e^{\alpha_2z},
	  	\end{equation}
	  	admits a finite order meromorphic solution $f(z)$ with $\lambda(f) <\rho(f)$ and $\lambda\left(\frac{1}{f}\right) <\rho(f)$, then $f(z)$ satsfies $\rho(f)=\deg Q =1$.\;Furthermore, $f(z)=A_1e^{\frac{\alpha_2}{n+1}z}$,\linebreak $Q(z)=\left(\alpha_1-\frac{\alpha_2}{n+1}\right)z + a_0$
	  	or $f(z)=A_2e^{\frac{\alpha_1}{n+1}z}$, $Q(z)=\left(\alpha_2-\frac{\alpha_1}{n+1}\right)z+a_0$, where $A_1,A_2, a_o \in \mathbb{C}\setminus\{0\}$.
	  	
	  	\end{theorem}
  	 The following examples demonstrate our result and cover the cases  $n=1,2,3$. The first example illustrates the case   $n=1$. 
  	 \begin{example}
  	 	$f(z)=e^{z}$ is a meromorphic solution of the differential-difference equation   $$f(z)f'(z)+2f(z+\log 2)e^{2z+\log 3}=1e^{2z}+12e^{3z},$$
  	 	where $n=1$, $\alpha_1=2$, $\alpha_2=3$ and $a_0= \log3$, then $f(z)=A_1e^{\frac{\alpha_1}{n+1}z}=e^{z}$,\linebreak $Q(z)=\left(\alpha_2-\frac{\alpha_1}{n+1}\right)z + a_0=2z+\log3$ and $\rho(f)= \deg Q=1.$
  	 \end{example}

  The next example gives the existence for $n=2$.
  
	  	\begin{example} $f(z)=e^{z}$ is a meromorphic solution of the differential-difference equation
	  		\begin{equation*}
	  			f^2(z)f'(z)+\frac{1}{7}f(z+\log 3)e^{-3z-\log3}=\frac{1}{7}e^{-z}+e^{3z},
	  		\end{equation*}
	  		where $n=2$, $\alpha_1=-2$, $\alpha_2=3$ and $a_0= \log3$, then $f(z)=A_1e^{\frac{\alpha_2}{n+1}z}=e^{z}$,\linebreak $Q(z)=\left(\alpha_1-\frac{\alpha_2}{n+1}\right)z + a_0=-3z-\log3$ and $\rho(f)= \deg Q=1.$
	  	\end{example}
	The next example demonstrates the existence of solutions in the case $n=3$.
	
	\begin{example}  $f(z)=e^{2z}$ is a meromorphic solution of the differential-difference equation
		\begin{equation*}
			f^3(z)f'(z)+\frac{1}{8}f(z-\log 4)e^{\left(\frac{-3}{2}z+\log16)\right)}=\frac{1}{8}e^{\frac{1}{2}z}+2e^{8z},
		\end{equation*}
		where $n=3$, $\alpha_1=\frac{1}{2}$, $\alpha_2=8$ and $a_0= \log16$, then $f(z)=A_1e^{\frac{\alpha_2}{n+1}z}=e^{2z}$,\linebreak $Q(z)=\left(\alpha_1-\frac{\alpha_2}{n+1}\right)z + a_0= \frac{-3}{2}z+\log16$ and $\rho(f)= \deg Q=1$.
	\end{example}
	
	 Throughout this paper, we assume that the reader is acquainted with the standard notations of Nevanlinna theory. For a meromorphic function $f$, $n(r,f)$, $N(r,f)$, $\Bar{N}(r,f) $, $m(r,f)$ and  $T(r,f)$ denote un-integrated counting function, integrated counting function, reduced counting function, proximity function and characteristic function respectively, see \cite{hay,ilpo,yanglo}.
	
	In Section $2$, we state  some definitions, lemmas  and results. In Section 3, we prove theorems.
	\medskip
	
	\section{\textbf{Auxiliary results}}
		To maintain self-containment, we briefly recall the definitions of the order of growth 
	$\rho(f)$,  hyper-order of growth $\rho_{2}(f)$  and the exponent of convergence of zeros $\lambda(f)$ for a meromorphic function $f(z)$.
	
	$$\rho(f)=\limsup_{r\to\infty}\frac{\log T(r,f)}{\log r},$$
	$$\rho_2(f)=\limsup_{r\to\infty}\frac{\log \log T(r,f)}{\log r},$$
	and
$$\lambda(f)=\limsup_{r\to\infty}\frac{\log n\left(r,\frac{1}{f}\right)}{\log r}.$$ 

	\bigskip
 
		For a meromorphic function $f$, Nevanlinna’s First Main Theorem asserts  that
	$$T\left(r,\frac{1}{f-a}\right)=T(r,f)+O(1),$$
	for all $a\in \mathbb{C}$,
	where $	O(1)$ denotes a bounded quantity depending only on $a$.
	\smallskip
	
	A meromorphic function $g(z)$ is a small function of $f(z)$ if $T(r,g)=S(r,f)$ and vice versa.
	For a meromorphic function $f(z)$,  $S(r,f)$ denotes the quantity satisfying $S(r,f)=o(T(r,f))$, as $r\to\infty$, outside  of a possible exceptional set $E$ (not necessarily same at each occurrence) of  finite linear measure. \\

 Borel’s lemma is a fundamental tool in applying Nevanlinna theory to complex differential equations.
	\begin{lemma}\label{imple1}(Borel's Lemma)\cite{cc}
		Suppose $f_1(z),f_2(z),...,f_n(z)(n\geq2)$ are meromorphic functions and $h_1(z),h_2(z),...,h_n(z)$ are entire functions satisfying:
		\begin{enumerate}
			\item $\sum_{i=1}^{\infty} f_i(z)e^{h_i(z)}\equiv 0$.
			\item For $1\leq i< k\leq n$, $h_i-h_k$ are  constants.
			\item  $1\leq i\leq n$, $1\leq m< k\leq n$,  $T(r,f_i(z))=o(T(r,e^{h_m-h_k}))$ as $r\rightarrow \infty$ outside of a set of finite linear measure.
		\end{enumerate}
		Then $f_i\equiv0\ (i=1,2,3,...,n)$.
	\end{lemma}
We now state a lemma that gives an estimate for the proximity function of the logarithmic derivative of a meromorphic function $f(z)$. 

	\begin{lemma}\label{imple2}\cite{ilpo}
		Suppose $f(z)$ is a transcendental meromorphic function and $k\geq1$ is an integer. Then 
		\begin{equation*}
			m\left(r,\frac{f^{(k)}}{f}\right)= S(r,f).
		\end{equation*}
	If $f$ is of finite order growth, then
	$$m\left(r,\frac{f^{(k)}}{f}\right)=O(\log r).$$
	\end{lemma}

The next lemma gives a sharp asymptotic estimate between $T(r, f(z + c))$ and $T(r, f)$, for meromorphic functions of finite order.
	\begin{lemma}\label{imple3}\cite{cf}
		Suppose $f(z)$ is a meromorphic function of finite order $\rho$ and $c$ is a non-zero complex constant. Then for every $\epsilon>0$,
		\begin{equation*}
			T(r,f(z+c))= T(r,f)+O(r^{\rho-1+\epsilon})+O(\log r).
		\end{equation*}
	\end{lemma}
Next lemma plays an important part in the study of complex differential-difference
equations.
\begin{lemma}\label{imple 4}\cite{yl}
Let $f(z)$ be  a transcendental meromorphic function, and $P(z,f)$, $Q(z,f)$ be two differential-difference polynomials of $f(z)$. If 
\begin{equation*}
	f^n(z) P(z,f)=Q(z,f)
\end{equation*}
holds, and if the total degree of $Q(z,f)$ in  $f(z)$ and its derivatives and their shifts is at most $n$, then
\begin{equation*} 
	m(r,P(z.f))=S(r,f),
	\end{equation*}
for all $r$ outside of a possible exceptional set of finite logarithmic measure.
\end{lemma}

The next lemma formulates the difference version of the lemma on logarithmic derivative lemma for finite-order meromorphic functions.
	\begin{lemma}\label{imple}\cite{cf}
		Suppose $f(z)$ is a meromorphic function with $\rho(f)<\infty$ and $c_1,c_2 \in \mathbb{C}$ such that $c_1\neq c_2$, then for each $ \epsilon>0$, we have
		\begin{equation*}
			m\left(r,\frac{f(z+c_1)}{f(z+c_2)}\right)=O(r^{\rho-1+\epsilon}).
		\end{equation*}
	\end{lemma}
	Next lemma estimates the characteristic function of an exponential polynomial $f$. This lemma can be seen in \cite{whl}.
	\begin{lemma}\label{imple5}\cite{whl}
		Suppose $f$ is an entire function given by
		$$f(z)=B_{0}(z)+B_{1}(z)e^{w_{1}z^{t}}+B_{2}(z)e^{w_{2}z^{t}}+...+B_{m}(z)e^{w_{m}z^{t}},$$
		where $B_{i}(z);0\leq i\leq m$ denote either exponential polynomial of degree $<t$ or polynomial in $z$, $w_{i};1\leq i\leq m$ denote the constants and $t$ denotes a natural number. Then
		$$T(r,f)=C(Co(W_{0}))\frac{r^{t}}{2\pi}+o(r^{t}),$$
		Here $C(Co(W_{0}))$ is the perimeter of the convex hull of the set $W_{0}=\{0,\overline{w}_{1},\overline{w}_{2},...,\overline{w}_{m}\}$.
	\end{lemma}
	\section{\textbf{Proof of main theorems}}
	\begin{proof}[\textbf{\underline{Proof of Theorem 1:}}]
		Suppose equation \eqref{1} has a meromorphic solution satisfying the conditions of Theorem $1$. First, we aim to prove that no meromorphic solution  $f(z)$ satisfying $\rho_2(f)<1$   exists for equation \eqref{1}  when $n\geq5$.
		
	\textbf{Case1:} \textbf{ $f(z)$ has at  least one pole}.\\ Let $z_0$ be a pole of $f(z)$ with multiplicity $ k\geq1$. If $c=0$, then
\begin{equation*}
	f^n +\omega f^{n-1}f' +p(z)f(z)=p_1e^{\alpha_1 z}+p_2e^{\alpha_2 z}.
\end{equation*}
Comparing the multiplicity  of the  pole $z_0$ on both sides gives us a contradiction. 
\\ If  $c\neq0$, then equation \eqref{1} indicates that  $z_0+c$ is a  pole of $f(z)$ with multiplicity at least $nk+1$. Now, replacing $z_0+c$ with $z$ in equation \eqref{1}, we get 
\begin{equation*}
	f^n(z_0+c) +\omega f^{n-1}(z_0+c)f'(z_0+c) +p(z_0+c)f(z_0+c)=p_1e^{\alpha_1 (z_0+c)}+p_2e^{\alpha_2 (z_0+c)}.
\end{equation*}
  Hence,  $z_0+2c$ is also a pole of $f(z)$ with multiplicity at least $n^2k+n+1$. By iterating the above procedure, we conclude that for any arbitrary integer $i\ (i\geq1)$, the point $z_0+ic$ is a pole of $f(z)$ with multiplicity at least $n^ik+n^{i-1}+\cdots+1$.
   Accordingly, for each integer $m$, we have
   \begin{equation*}
   	n(m|c|+|z_0|+1,f)\geq k+\sum_{i=1}^{m} [ n^ik+n^{i-1}+\cdots+1 ].
   \end{equation*}
	Thus, 
	\begin{align*}
	\rho_2(f) \geq \lambda_2 \left( \frac{1}{f}\right) &=\limsup\limits_{r\to \infty}\frac{\log\log n(r,f)}{\log r}\\& \geq\limsup\limits_{m\to \infty}\frac{\log\log n(m|c|+|z_0|+1,f)}{\log m|c|+|z_0|+1}\\&\geq\limsup\limits_{m\to \infty} \frac{\log\log n^m}{m}\\&=1, 
	\end{align*}
which is a contradiction to the given condition $\rho_2<1$.\\
	 \textbf{Case2:} \textbf{ $f(z)$  has no pole i.e. $f(z)$ is an entire function.}\\
	 If $f(z)$ is a polynomial, then comparing growth of both sides, we get that order of growth of L.H.S is $0$, while R.H.S has $1$, which is contradictory. Hence, $f(z)$ is transcendental. 
	 Let $H(z)=p(z)f(z+c)$ and $Q(z)= f^n+\omega f^{n-1}f'$, then  equation \eqref{1} 	 
	 becomes
	 \begin{equation}\label{2}
	 	Q+H=p_1e^{\alpha_1z}+p_2e^{\alpha_2z}. 
 		 \end{equation}
	Differentiating  equation \eqref{2}, we get 
	\begin{equation}\label{3}
		 Q'+H'=p_1\alpha_1e^{\alpha_1z}+p_2\alpha_2e^{\alpha_2z}.
		\end{equation}
Now, eliminating $e^{\alpha_1z}$ from \eqref{2} and \eqref{3}
	 \begin{equation}\label{4}
	 	 \alpha_1Q-Q'+ \alpha_1H-H'=p_2(\alpha_1-\alpha_2)e^{\alpha_2z}.
	 \end{equation}
 Differentiating  equation \eqref{4}, we get
 \begin{equation}\label{5}
 	\alpha_1Q'-Q''+ \alpha_1 H'-H''=p_2(\alpha_1-\alpha_2)\alpha_2e^{\alpha_2z}.
 \end{equation}
 Use  equations \eqref{4} and \eqref{5} to eliminate $e^{\alpha_2z}$, 
\begin{equation}\label{6}
	\alpha_1\alpha_2Q-(\alpha_1+\alpha_2)Q'+Q''=-(\alpha_1\alpha_2H-(\alpha_1+\alpha_2)H'+H'').
\end{equation} 
 Note that,
 \begin{equation*}
 	Q'= nf^{n-1}f'+(n-1)\omega f^{n-2}(f')^2 +\omega f^{n-1}f'',
 \end{equation*}
\begin{align*}
 	Q''=n(n-1)f^{n-2}(f')^2+nf^{n-1}f''+(n-1)(n-2)\omega f^{n-3}f'^3\\  +2(n-1)\omega f^{n-2}f'f''+\omega f^{n-1}f'''+(n-1)\omega f^{n-2}f'f''.
 	\end{align*}
 Substituting in equation \eqref{6}, we get
 \begin{equation}\label{7}
 	f^{n-3}\phi=-(\alpha_1\alpha_2H-(\alpha_1+\alpha_2)H'+H''),
 \end{equation} 
 where
 \begin{align}\label{A}
	\phi=\alpha_1\alpha_2f^3+(\omega\alpha_1\alpha_2-n(\alpha_1+\alpha_2))f^2f'+(n-\omega(\alpha_1+\alpha_2))f^2f''\\ \nonumber+\omega f^2f'''+ (n-1)(n-\omega(\alpha_1+\alpha_2))ff'^2+3\omega(n-1)ff'f'''\\ \nonumber +\omega(n-1)(n-2)f'^3,
 \end{align}
is entire as $f(z)$ is an entire function. Thus, $T(r,\phi)= m(r,\phi)$.\\
\begin{enumerate}
 \item If $\phi \not\equiv  0$, 
 then from equation \eqref{7}, $n\geq5$ and using Lemma \ref{imple 4}, we get 
	\begin{align*}
	m(r,\phi)= S(r,f),\\ m(r,f \phi)=S(r,f).
	\end{align*}
Using these above equations and First Fundamental theorem of Nevanlinna, we get
\begin{align*}
	T(r,f)=m\left(r,\frac{f\phi}{\phi}\right)=m(r,f\phi)+m\left(r,\frac{1}{\phi}\right)=S(r,f),
\end{align*}
which is a contradiction. Therefore, $\phi \equiv 0$.
\item  If $\phi  \equiv 0$, then equation \eqref{7} implies 
$\alpha_1\alpha_2H-(\alpha_1+\alpha_2)H'+H''=0$, and solution will be of the form 
\begin{equation*}
H(z)=\alpha_1 e^{\alpha_1z}+\alpha_2e^{\alpha_2z}.
\end{equation*}
Using $H(z)=p(z)f(z+c)$, we have 
\begin{equation}\label{8}
	f(z)=\gamma_1 e^{\alpha_1z}+\gamma_2e^{\alpha_2z},
	\end{equation}
	where $\gamma_1=\frac{\alpha_1e^{-\alpha_1c}}{p(z-c)}$, and $\gamma_2=\frac{\alpha_2e^{-\alpha_2c}}{p(z-c)}$.\\
	
In the similar manner, with the help of  equation \eqref{6} i.e.,
\begin{equation*}
\alpha_1\alpha_2Q-(\alpha_1+\alpha_2)Q'+Q''=0,
\end{equation*}
we have solution of the form 
\begin{equation*}
	Q(z)=m_1 e^{\alpha_1z}+m_2e^{\alpha_2z},
\end{equation*}
where $m_1$ and $m_2$ are constants. Using the fact that  $Q(z)=f^n+\omega f^{n-1}f'$   and  $f(z)=\gamma_1 e^{\alpha_1z}+\gamma_2e^{\alpha_2z}$ , we have
\begin{align}\label{9}
(\gamma_1e^{\alpha_1z}+\gamma_2e^{\alpha_2z})^{n-1}[\gamma_1e^{\alpha_1z}+\gamma_2e^{\alpha_2z}+\omega & \nonumber (\alpha_1\gamma_1e^{\alpha_1z}+\alpha_2\gamma_2e^{\alpha_2z})]\\&=m_1e^{\alpha_1z}+m_2e^{\alpha_2z}.
\end{align}
Dividing equation \eqref{9}  by $e^{\alpha_2z}$ and solve to get
\begin{align*}
\\&	(\gamma_1e^{\alpha_1z}+\gamma_2e^{\alpha_2z})^{n-1}[\gamma_1(\omega \alpha_1+1)e^{\alpha_1z}+\gamma_2  (\omega \alpha_2+1)e^{\alpha_2z}]=m_1e^{(\alpha_1-\alpha_2)z}+m_2,\\&
\gamma_1(\omega \alpha_1+1) \left[\sum_{i=0}^{n-2} \ ^{n-1}C_{i} (\gamma_1)^{i} (\gamma_2)^{n-i-1} e^{(i+1)\alpha_1z+(n-i-2)\alpha_2z} +\gamma_1^{n-1}e^{n\alpha_1z-\alpha_2z}\right]\\& + \gamma_2(\omega \alpha_2+1) \left[\sum_{i=0}^{n-1} \ ^{n-1}C_{i} (\gamma_1)^{i} (\gamma_2)^{n-i-1} e^{i\alpha_1z+(n-i-1)\alpha_2z} \right]=m_1e^{(\alpha_1-\alpha_2)z}+m_2.
\end{align*}
\end{enumerate}
Since $\frac{\alpha_1}{\alpha_2}\neq n$, $\alpha_1\neq \alpha_2$, and $\frac{\alpha_2}{\alpha_1}\neq n$, applying Borel's lemma to the above equation yields
$\gamma_1^{n-1}=0$, i.e. $\gamma_1=0$, which leads to a contradiction.
Similarly, dividing \eqref{9} by $e^{\alpha_1z}$ gives $\gamma_2=0$, which is again a contradiction.
Therefore, in both cases, we arrive at a contradiction, which implies that equation \eqref{1} does not admit a meromorphic solution with $\rho_2(f) < 1$ when $n \geq 5$.\\
Now, we will prove the left part of the theorem i.e $\Bar{\lambda}(f)=\rho(f)$.
We can rewrite equation \eqref{7} for $n=4$ as,
\begin{equation}\label{10}
f\phi=-(\alpha_1\alpha_2H-(\alpha_1+\alpha_2)H'+H'').
\end{equation}
If $\phi\equiv 0$, then by proceeding in similar manner as in case $2$, we get a contradiction.\\
If   $\phi\not \equiv 0$, then applying Clunie's lemma \ref{imple 4} to equation \eqref{10} gives that 
$m(r,\phi)=S(r,f)$ and hence
\begin{equation}\label{S}
	T(r,\phi)=m(r,\phi)+N(r, \phi)= S(r,f).
\end{equation} 
Using the First Fundamental theorem of Nevanlinna, we also have 
\begin{equation}\label{12}
	T\left(r,\frac{1}{\phi}\right)=m\left(r,\frac{1}{\phi}\right)+N\left(r, \frac{1}{\phi}\right)= S(r,f).
\end{equation}
From equation \eqref{A}, we have 
\begin{align}\label{B}
	\phi=\alpha_1\alpha_2f^3+Af^2f'+Bf^2f''+\omega f^2f'''+ Cff'^2+3\omega(n-1)ff'f'''\\ \nonumber +\omega(n-1)(n-2)f'^3,
\end{align} 
where $A= \omega\alpha_1\alpha_2-n(\alpha_1+\alpha_2), 
B= n-\omega(\alpha_1+\alpha_2)$ and $  C= (n-1)(n-\omega(\alpha_1+\alpha_2)).$
Hence, equation \eqref{B} becomes 
\begin{align*}
	\frac{\phi}{f^3}=\alpha_1\alpha_2+A\frac{f}{f'}+B\frac{f''}{f}+\omega \frac{f'''}{f}+ C\frac{f'^2}{f^2}+3\omega(n-1)\frac{f'f'''}{f^2} +\omega(n-1)(n-2)\frac{f'^3}{f^3},
\end{align*} 
using Lemma \ref{imple2}, we have 
\begin{equation}\label{14}
 m\left(r,\frac{\phi}{f^3}\right)= S(r,f).
\end{equation}  
Next, with the help of  the lemma  of logarithmic derivatives  \ref{imple2} and above obtained expression, we have
\begin{align}\label{C}
 	3m\left(r,\frac{1}{f}\right) & \leq m\left(r,\frac{1}{\phi}\right)+3m\left(r,\frac{\phi}{f^3}\right)
 \\  \nonumber	& \leq m\left(r,\frac{1}{\phi}\right)+S(r,f).
 	\end{align}
  Furthermore, equation \eqref{A} implies  if $z_0$ is a multiple zero of $f$, then $z_0$ is also a zero of $\phi$. Therefore, it follows 
  \begin{equation}\label{15}
  N\left(r,\frac{1}{f}\right) \leq \bar{N}\left(r,\frac{1}{f}\right)+N\left(r, \frac{1}{\phi}\right)+S(r,f).
  \end{equation}
 With the help of the First Fundamental theorem and equations \eqref{12}, \eqref{C} and \eqref{15}, we have
\begin{align*}
		T(r,f) &= T\left(r,\frac{1}{f}\right)+ S(r,f) \\ & =m\left(r,\frac{1}{f}\right)+ N\left(r,\frac{1}{f}\right) +S(r,f) \\ & \leq \frac{1}{3}m\left(r,\frac{1}{\phi}\right)+ \bar{N}\left(r,\frac{1}{f}\right)+N\left(r,\frac{1}{\phi}\right) +S(r,f) \\ & \leq \bar{N}\left(r,\frac{1}{f}\right) +S(r,f),
	 \end{align*}
 which implies $\rho(f)\leq \Bar{\lambda}(f)$.
Hence, 
$\rho(f)= \Bar{\lambda}(f)$.
\end{proof}
\begin{proof}[\textbf{\underline{Proof of Theorem 2:}}]  We prove this theorem by contradiction. Let us assume that there exist $f(z)$, a meromorphic solution of equation \eqref{1}, and $\lambda(f)<\rho(f)$, then $ N(r,\frac{1}{f}) = S(r,f)$.
From Theorem 1, using equation \eqref{7} for $n=4$, we have 
\begin{equation*}
	f\phi=-(\alpha_1\alpha_2H-(\alpha_1+\alpha_2)H'+H'').
\end{equation*}
	Using Lemma \ref{imple 4} and equation \eqref{14} we have, 
	\begin{equation}\label{17}
		m(r,\phi)= S(r,f),     	m\left(r,\frac{\phi}{f^3}\right) = S(r,f).
	\end{equation}
	  Also $ N\left(r,\frac{1}{f}\right) = S(r,f)$, which implies
	\begin{equation}\label{18}
		 N\left(r,\frac{\phi}{f^3}\right) =  N(r,\phi)+N\left(r,\frac{1}{f^3}\right) =  S(r,f).
	\end{equation}
	If   $\phi\not\equiv 0$, then using equation \eqref{S} we have,
	\begin{align*}
		3T(r,f)=T(r,f^3) &= T\left(r,\frac{1}{f^3}\right)+O(1) \\&\leq T\left(r,\frac{\phi}{f^3}\right)+ T\left(r,\frac{1}{\phi}\right) +S(r,f) \\&\leq T(r,\phi) +S(r,f) \\& =S(r,f),
		\end{align*}
	 which is a contradiction.
	If $\phi \equiv 0$, then  we can proceed as in the previous theorem to arrive at the desired conclusion. Hence, no such solution exists for $n=4$. \\
Now, for $n=3$ equation \eqref{7} becomes 
	\begin{equation*}
	\phi=-(\alpha_1\alpha_2H-(\alpha_1+\alpha_2)H'+H'').
\end{equation*}
Above equation implies
\begin{equation*}
	m\left(r,\frac{\phi}{f}\right)= m\left(r,\frac{\alpha_1\alpha_2H-(\alpha_1+\alpha_2)H'+H''}{f}\right)=  S(r,f).
\end{equation*}
From equation \eqref{14}, we have 	$m\left(r,\frac{\phi}{f^3}\right) = S(r,f)$.
 For $\phi\not\equiv 0$, with the help of  $ N\left(r,\frac{1}{f}\right) = S(r,f)$ and above two equations, we have 
 	\begin{align*}
 	3T(r,f)=T(r,f^3) &= m\left(r,\frac{1}{f^3}\right)+ N\left(r,\frac{1}{f^3}\right) +O(1) \\&\leq m\left(r,\frac{\phi}{f^3}\right)+ m\left(r,\frac{1}{\phi}\right)+3 N\left(r,\frac{1}{f}\right) + O(1) \\&\leq T(r,\phi) +S(r,f) \\& = m(r,\phi)+S(r,f)  \\& \leq  m\left(r,\frac{\phi}{f}\right) + m(r,f) +S(r,f)\\& \leq T(r,f)+S(r,f),
 \end{align*}
and hence $2T(r,f)\leq S(r,f)$, which is a contradiction.
If $\phi \equiv0 $, then we can easily get a contradiction as in previous cases.
Hence, no such solution exist for this case also. 
Thus, theorem is proved.  
\end{proof}
		\begin{proof}[\textbf{\underline{Proof of Theorem 4:}}] 
		For $n\geq 4$, theorem  is proved. So, we will prove the result for $n=1,2,3$. Let $f(z)$ be a finite order meromorphic transcendental  solution with $\lambda(f)<\rho(f)$ and $\lambda(\frac{1}{f})<\rho(f)$. First, we will prove $\rho(f)=1$.\\
		\textbf{Case 1.} If $\rho(f)<1$, then using  the First Fundamental theorem of Nevanlinna and applying Lemmas \ref{imple2}, \ref{imple3},  and \ref{imple} to equation \eqref{z}, we get
		
	\begin{align*}
		T(r, e^{Q(z)}) &=T\left(r,\frac{p_1e^{\alpha_1z}+p_2e^{\alpha_2z}-f^nf'}{q(z)f(z+c)}\right) \\&\leq T(r,p_1e^{\alpha_1z}+p_2e^{\alpha_2z})+T\left(r,\frac{f^nf'f}{f}\right) +T\left(r,\frac{1}{q(z)f(z+c)}\right) \\& \leq T(r,p_1e^{\alpha_1z}+p_2e^{\alpha_2z})+(n+1)T(r,f) +T\left(r,\frac{1}{f}\right) +O(\log r) \\& \leq T(r,p_1e^{\alpha_1z}+p_2e^{\alpha_2z}) +(n+2)T(r,f) +S(r,p_1e^{\alpha_1z}+p_2e^{\alpha_2z}) \\& \leq  \max\{ \rho (p_1e^{\alpha_1z}+p_2e^{\alpha_2z}),\rho(f) \}.
	\end{align*}
This gives $\deg (Q(z))\leq1$, but our assumption is that $\deg (Q(z))\geq 1$. Therefore, $\deg (Q(z))=1$, say $Q(z)=lz+m$; $l\neq0$. So, equation \eqref{z} becomes
\begin{equation*}
	f^n(z)f'(z)+ q(z)f(z+c)e^{lz+m}=p_1e^{\alpha_1z}+p_2e^{\alpha_2z}.
\end{equation*}
Differentiating above equation, we get
\begin{equation*}
	f^nf''+ n f^{(n-1)}f'f' + \beta(z)e^{lz+m} = p_1 \alpha_1e^{\alpha_1z}+p_2\alpha_2e^{\alpha_2z},
\end{equation*}
where $\beta(z)=q'(z)f(z+c)+f'(z+c)q(z) +q(z)f(z+c)l$.\\

Eliminating $e^{\alpha_2}$ from above two equations, we obtain
\begin{equation}
f^nf''+n f^{(n-1)}(f')^2-\alpha_2f^nf'+( \beta(z)-\alpha_2q(z)f(z+c))e^{lz+m}= p_1 (\alpha_1-\alpha_2)e^{\alpha_1z}.
\end{equation}
\textbf{Case 1(i).} If $l \neq \alpha_1$, then using Borel's lemma \ref{imple1}, we have $p_1(\alpha_1-\alpha_2)=0$, implying $\alpha_1=\alpha_2$, which is a contradiction to the assumption.\\
\textbf{Case 1(ii).} If $l=\alpha_1$, then
\begin{equation*}
	f^nf''+n f^{(n-1)}(f')^2-\alpha_2f^nf'+( \beta(z)-\alpha_2q(z)f(z+c)e^m-p_1 (\alpha_1-\alpha_2))e^{\alpha_1z}=0.
\end{equation*}
Applying Borel's lemma \eqref{imple1}, we get $	f^nf''+n f^{(n-1)}(f')^2-\alpha_2f^nf'\equiv0$ or
\begin{equation}\label{23}
	n(f'^2)+ff''-\alpha_2ff' \equiv 0.
\end{equation}
Recall,
\begin{equation*}
	\frac{f''}{f}= \left(\frac{f'}{f}\right)'+  \left(\frac{f'}{f}\right)^2
\end{equation*}
 and use it in equation \eqref{23} to get a Bernoulli equation,
\begin{equation*}
	\alpha_2s-s'-(n+1)s^2=0,
\end{equation*} 
where $s=\frac{f'}{f}$.
	
Now, from routine calculations we have two solutions  $s_1=0$ and $s_2=\frac{\alpha_2}{n+1}$.
\textbf{Case1(ii)(a):}
 If $s\neq s_1$ and $s\neq s_2$, then 
 \begin{equation*}
 \frac{\alpha_2}{n+1}\left(\frac{s'}{s}-\frac{s'}{s-\frac{\alpha_2}{n+1}}\right)=1. 	
 \end{equation*}
	Integrating	w.r.t $z$, we have
	\begin{equation*}
		\log s- \log\left({s-\frac{\alpha_2}{n+1}}\right)=\frac{n+1 }{\alpha_2}z+c,
\end{equation*}
\begin{equation}
		\log \left(\frac{s}{{s-\frac{\alpha_2}{n+1}}}\right)=\frac{n+1 }{\alpha_2}z+c,
\end{equation}
	where $c$ is a constant.
	Above equation implies
	\begin{equation}\label{25}
		\left(\frac{s}{s-\frac{\alpha_2}{n+1}}\right)=  e^{A'},
	\end{equation}
where $A'= \frac{n+1}{\alpha_2}z+c$.		
 Solving equation\eqref{25}, we have 
 \begin{equation*}
 	s= \frac{f'}{f}=\frac{\alpha_2}{n+1}+\frac{\frac{\alpha_2}{n+1}}{e^{A'-1}}.
 \end{equation*}	
		Observe that zeroes of $e^{A'}-1$ are zeroes of $f$. Let $z_0$ be zero of $e^{A'}-1$  with multiplicity $m$, then 
		\begin{equation}
			Res\left(\frac{f'}{f},z_0\right)= Res \left(\frac{\alpha_2}{n+1}+\frac{\frac{\alpha_2}{n+1}}{e^{A'-1}}, z_0\right)= \frac{1}{n+1},
		\end{equation}
which is a contradiction.\\
	\textbf{Case 1(ii)(b):} If $s=s_1=0$, then $\frac{f'}{f} =0$, gives $f=0$, a contradiction.\\
		\textbf{Case 1(ii)(c):} If $s=s_2=\frac{\alpha_2}{n+1}$, then 
		\begin{equation*}
			\frac{f'}{f}=\frac{\alpha_2}{n+1}.
		\end{equation*}
	On solving, we get 
	\begin{equation*}
			f= e^{\frac{\alpha_2}{n+1}z+e}.
		\end{equation*}
Thus, $\rho(f)=1$ which is a contradiction to our assumption that $\rho(f)<1$.\\
		Hence, $\rho(f)\geq1$. Assume $\rho(f)>1$, and for simplicity define 
		 $$S(z) = p_1 e^{\alpha_1 z} + p_2 e^{\alpha_2 z} \qquad \text{and} \qquad	R(z) = q(z)f(z+c).$$
		  With these notations, equation \eqref{z}  can be rewritten as
		 \begin{equation*}
		 	f^{n}(z)f'(z) + R(z)e^{Q(z)} = S(z).
		 	\end{equation*}
		 	Differentiating the above  equation, we get 
		 	 		\begin{equation*}
		 	 			nf^{n-1}(f')^2+f^nf''+(	R'+Q'R)e^{Q(z)}=S'.
		 	 			\end{equation*}
		 Eliminating $e^{Q(z)}$ from above two equations,
		 \begin{equation}\label{27}
		 	f^{n-1}\psi =R(z)S'(z)-P(z)S(z),
		 	\end{equation}
		where
		\begin{equation*}
			 P(z)=R'(z)+Q'(z)R(z)
			\end{equation*}
		and
		\begin{equation}\label{28}
		\psi(z)= nR(f')^2+Rff''-(R'+Q'R)ff'.
		\end{equation}
	\begin{enumerate}[(i)]
		
		\item Let $n=3$, then applying analogue of Clunie lemma \ref{imple 4} to equation \eqref{27}, we have
		\begin{equation*}
			m(r,\psi)=S(r,f). 
			\end{equation*}
		Also, above equation \eqref{27} can be written as 
		\begin{equation}
			f(f\psi)=RS'-PS,
		\end{equation}
		and hence again applying Clunie lemma  \ref{imple 4}  we get 
		\begin{equation*}
		 m(r,f\psi)=S(r,f).
		\end{equation*}
			Now, with the help of obtained equations we get the following expression
			\begin{equation*}
				T(r,f)=T\left( r,\frac{f\psi}{\psi}\right)= T(r,f\psi)+T\left(r,\frac{1}{\psi}\right) =S(r,f),
			\end{equation*}
	which is a contradiction.	
		
	\item  Let $n=2$, then applying Lemma \ref{imple3} to equation \eqref{27}, we get
		\begin{equation}\label{29}
			m(r, \psi)=S(r,f).
		\end{equation}
		Also  $\lambda(f)<\rho(f)$  
	implies $$N\left(r,\frac{1}{f}\right)=S(r,f),$$
		hence
		\begin{equation}\label{30}
			N\left(r, \frac{\psi}{f^3}\right)=N\left( r, \frac{1}{f^3}\right) +N(r,\psi) =3N\left( r, \frac{1}{f}\right) =S(r,f).
		\end{equation}
		Applying Lemmas \ref{imple3} and \ref{imple} to equation \eqref{28}, we get
		\begin{equation}\label{eq3.7}
			m\left( r,\frac{\psi}{f^3}\right) =S(r,f).
		\end{equation}
		If $\psi\not\equiv 0$, then using equations \eqref{29}, \eqref{30} \& \eqref{eq3.7}, and the First Fundamental theorem of Nevanlinna, we get
		\begin{align*}
			3T(r,f)=T(r,f^3)&=T\left( r, \frac{1}{f^3}\right) +O(1)\\
			&\leq T\left( r, \frac{\psi}{f^3}\right) +T\left( r,\frac{1}{\psi}\right) +O(1) \\
			&=S(r,f),
		\end{align*}
		which is a contradiction.\\
		If $\psi\equiv 0$, then from equation \eqref{27}, we get
		$R(z)S'(z)-P(z)S(z)\equiv 0$.
		This gives
		\begin{align*} &R(z)S'(z)-(R'(z)+Q'(z)R(z))S(z)\equiv 0\\
		\text{i.e. }\ &\frac{S'(z)}{S(z)}-Q'(z)-\frac{R'(z)}{R(z)}=0\\
		\text{i.e. }\ &\frac{S'(z)}{S(z)}-Q'(z)-\frac{q'(z)}{q(z)}-\frac{f'(z+c)}{f(z+c)}=0.\end{align*}
		On integrating above equation, we get
		\begin{equation}\label{33}
			q(z)f(z+c)e^{Q(z)}=\frac{1}{C}S(z)=\frac{1}{C}(p_1e^{\alpha_{1}z}+p_1e^{\alpha_{2}z}),
		\end{equation} where $C$ is an integrating constant.
		As  $f$ is a finite order transcendental entire solution satisfying $\lambda(f)<\rho(f)$  via the Hadamard factorisation theorem $f$ must be of the form \begin{equation}\label{34}
			f(z)=g(z)e^{h(z)},
		\end{equation}
		 where $h(z)$ is a polynomial such that $\rho(f)=\deg (h)>1$ and $g(z)$ is the canonical product of zeros of $f(z)$ with $\lambda(f)=\rho(g)<\rho(f)$.\\
		Using equations \eqref{33} and \eqref{34} to the equation \eqref{z} we have 
			$$ [g^n(z)g'(z)+g^{n+1}(z)h'(z)]e^{(n+1)h(z)}=(1-\frac{1}{C})(p_{1}e^{\alpha_{1}z}+p_{2}e^{\alpha_{2}z}).$$
		On applying Lemma \ref{imple5} to the above equation, we get that the order of growth of the left side is greater than $1$, while the order of growth of the right side is exactly $1$. This is a  contradiction, hence $\rho(f)=1$.
		
	\item Let $n=1$, then applying Lemmas \ref{imple2} and \ref{imple} to equations \eqref{27} and \eqref{28} gives
		\begin{equation}\label{4.3}
			m\left(r,\frac{\psi}{f}\right)=m\left(r,\frac{R(z)S'(z)-(R'(z)+Q'(z)R(z))S(z)}{f}\right)=S(r,f)
		\end{equation}
		and 
		\begin{equation}\label{4.4}
			m\left( r,\frac{\psi}{f^{3}}\right) =S(r,f).
		\end{equation}
	Also, $\lambda(f)<\rho(f)$ gives
		\begin{equation}\label{4.5}
			N\left( r,\frac{1}{f}\right) =S(r,f).
		\end{equation}
		If $\psi\not\equiv 0$, then using the First Fundamental theorem of Nevanlinna and equations \eqref{4.3}, \eqref{4.4} \& \eqref{4.5}, we have
		\begin{align*}
			3T(r,f)=T(r,f^{3})&=m\left( r,\frac{1}{f^{3}}\right) +N\left( r,\frac{1}{f^{3}}\right) +O(1)\\
			&\leq m\left( r,\frac{\psi}{f^{3}}\right) +m\left( r,\frac{1}{\psi}\right) +3N\left( r,\frac{1}{f}\right) +O(1)\\
			&\leq T(r,\psi)+S(r,f)\\
			&= m(r,\psi)+S(r,f)\\
			&\leq m\left( r,\frac{\psi}{f}\right) +m(r,f)+S(r,f)\\
			&\leq T(r,f)+S(r,f).
		\end{align*}
		This gives $2T(r,f)=S(r,f)$, which is a contradiction.\\
		If $\psi\equiv 0$, then proceeding in the similar manner as done in (ii), we get the same contradiction. Thus, $\rho(f)=1$.\\
		Next, to prove $\deg(Q)=1$ and other conclusion, we follow  same techniques as done in \cite[Proof of Theorem 7]{jh}.
	\end{enumerate}
		
	\end{proof}
		
		\section*{\textbf{Acknowledgement}}
		The author is thankful  to her supervisor Professor Sanjay Pant for his valuable comments and suggestions to improve the readibility of the paper.

	\end{document}